\documentclass{amsart}
\usepackage{amsfonts}
\usepackage{latexsym}
\usepackage{amssymb}
\usepackage{amsmath}
\usepackage{color}
\usepackage{bbm}
\usepackage{tikz}
\usepackage{enumerate}

\newcommand{\R}{\mathbb R}
\newcommand{\N}{\mathbb N}

\newcommand{\E}{\mathbb E}
\newcommand{\Pro}{\mathbb P}
\newcommand{\dif}{\,\mathrm{d}}

\newtheorem{thm}{Theorem}[section]
\newtheorem{cor}[thm]{Corollary}
\newtheorem{lemma}[thm]{Lemma}

\newtheorem{proposition}[thm]{Proposition}

\theoremstyle{remark}
\newtheorem*{rmk}{Remark}

\begin{document}


\title[Random variables corresponding to Orlicz norms]{On the Distribution of Random variables corresponding to Musielak-Orlicz norms}

\author[D. Alonso-Guti\'errez]{David Alonso-Guti\'errez}
\address{Departamento de Matem\'aticas, Universidad de Murcia, Campus de Espinar\-do, 30100-Murcia, Spain}
\email{davidalonso@um.es}

\author[S. Christensen]{S\"oren Christensen}
\address{Mathematisches Seminar, Christian Albrechts University Kiel, Ludewig-Meyn-Stra{\ss}e~4, 24098 Kiel, Germany}
\email{christensen@math.uni-kiel.de}

\author[M. Passenbrunner]{Markus Passenbrunner}
\address{Institute of Analysis, Johannes Kepler University Linz,
Altenbergerstra{\ss}e 69, 4040 Linz, Austria} \email{markus.passenbrunner@jku.at}

\author[J. Prochno]{Joscha Prochno}
\address{Institute of Analysis, Johannes Kepler University Linz,
Altenbergerstra{\ss}e 69, 4040 Linz, Austria} \email{joscha.prochno@jku.at}


\thanks{The first author is partially supported by MICINN project MTM2010-16679,
MICINN-FEDER project MTM2009-10418 and ``Programa de Ayudas a Grupos de
Excelencia de la Regi\'on de Murcia'', Fundaci\'on S\'eneca,
04540/GERM/06. The third and fourth author are supported by the Austrian Science Fund,
FWF project P23987 ``Projection operators in
Analysis and geometry of classical Banach spaces''.}

\date{\today}

\begin{abstract}
Given a normalized Orlicz function $M$ we provide an easy formula for a distribution such that, if $X$ is a random variable distributed accordingly and $X_1,\ldots,X_n$ are independent copies of $X$, then
\[
\frac{1}{C_p} \|x\|_M \leq \mathbb E \|(x_iX_i)_{i=1}^n\|_p \leq C_p\|x\|_M,
\]
where $C_p$ is a positive constant depending only on $p$. In case $p=2$ we need the function $t\mapsto tM'(t) - M(t)$ to be $2$-concave and as an application immediately obtain an embedding of the corresponding Orlicz spaces into $L_1[0,1]$.
We also provide a general result replacing the $\ell_p$-norm by an arbitrary $N$-norm.
This complements some deep results obtained by Gordon, Litvak, Sch\"utt, and Werner in \cite{GLSW1}. We also prove a result in the spirit of \cite{GLSW1} which is of a simpler form and easier to apply. All results are true in the more general setting of Musielak-Orlicz spaces. 
\end{abstract}

\maketitle

\section{Introduction}

In their outstanding work \cite{KS1}, Kwapie\'n and Sch\"utt obtained beautiful and strong combinatorial inequalities in connection with Orlicz norms that were then used to study certain invariants of Banach spaces (see also \cite{KS2}). The new tool not only allowed them to compute the positive projection constant of a finite-dimensional Orlicz space, but also led to a characterization of the symmetric sublattices of $\ell_1(c_0)$ and the finite-dimensional symmetric subspaces of $\ell_1$. The method was later used in \cite{IS} to determine $p$-absolutely summing norms and was extended by Raynaud and Sch\"utt to infinite-dimensional Banach spaces in \cite{RS} (see also \cite{S2} for applications to Lorentz spaces). In some special cases, the combinatorial expressions were already considered by Gluskin in \cite{G} (see also \cite{S1}). Quite recently, in \cite{PS}, the tools were generalized to obtain new results on the local structure of the classical Banach space $L_1$.

In the great paper \cite{GLSW1}, building upon the combinatorial results from \cite{KS1} and \cite{KS2}, Gordon, Litvak, Sch\"utt and Werner were able to obtain even more general results in the continuous setting. They proved that, if $N$ is an Orlicz function and $X_1,\ldots,X_n$ are independent copies of a random variable $X$, then $\mathbb E \| (x_iX_i)_{i=1}^n \|_N$ is of the order $\|x\|_M$ where $M$ depends on $N$ and the distribution of $X$. This result, of course, is already interesting from a purely probabilistic point of view and was later used by the authors in \cite{GLSW3} to obtain estimates for various parameters associated to the local theory of convex bodies. It also initiated further research and led to beautiful results on order statistics \cite{GLSW2,GLSW}. Recently, in the series of papers \cite{AGP2, AGP, AGP3}, these results were also successfully used to study geometric functionals corresponding to random polytopes.

A natural question that arises is whether the converse is true, i.e., given Orlicz functions $M$ and $N$, can we provide a formula for a distribution so that, if $X_1,\ldots,X_n$ are independent copies of an accordingly distributed random variable $X$, then $\mathbb E \| (x_iX_i)_{i=1}^n \|_N$ is of the order $\|x\|_M$. This is one part of the motivation for our work and we will answer this question in the affirmative. The ``natural'' candidate for the distribution is deduced from a new simpler version of a result from \cite{GLSW1} that we prove here. In the special case of $N(t)=t^p$ we give very easy formulas for the distribution of the random variables depending on the Orlicz function $M$, provided $M$ satisfies a certain condition depending on the parameter $p$. For $p=2$, this condition amounts to the $2$-concavity of $t\mapsto tM'(t) - M(t)$.

In his beautiful paper \cite{S} Sch\"utt proved that, if $M$ is equivalent to a $2$-concave Orlicz function, then the spaces $\ell_M^n$, $n\in\N$ embed uniformly into $L_1$ (see also \cite{BDC} and \cite{P1}). The proof is quite technical and based on combinatorial inequalities, some of them first appeared in the joint work \cite{KS1,KS2} with Kwapie\'n. Given a $2$-concave Orlicz function $M$ with some additional properties, he provided an explicit formula to obtain a sequence $a_1,\ldots,a_n$ of positive real numbers so that for all $x\in\R^n$
\[
c_1 \|x\|_M \leq \frac{1}{n!} \sum_{\pi \in \mathfrak{S}_n} \left( \sum_{i=1}^n |x_ia_{\pi(i)}|^2\right)^{1/2} \leq c_2 \|x\|_M,
\]
where $\mathfrak{S}_n$ is the set of all permutations of the numbers $\{1,\ldots,n\}$ and $c_1,c_2$ are absolute constants (see Theorem 2 in \cite{S}). Khintchine's inequality then implies that these Orlicz spaces embed uniformly into $L_1$. Unfortunately, the formula is rather complicated and it is  non-trivial to calculate the Orlicz function. This, in fact, shall be the other part of our motivation. The converse result we obtain for $p=2$, where we need $t\mapsto tM'(t) - M(t)$ to be $2$-concave, immediately implies that these Orlicz spaces $\ell_M^n$, $n\in \N$ are uniformly isomorphic to subspaces of $L_1$. Although it seems we need a somehow stronger assumption on $M$, the inversion formula we obtain is much simpler and easier to apply. The result might also be useful in finding new and easily verifiable characterizations for more general classes of subspaces of $L^1$.

We provide here two different approaches to prove the converse results (for $\ell_p$-norms and general $N$-norms) where in each one of them conditions on $M$ naturally appear. Even more, if $p=2$ and we do not assume the $2$-concavity of $t\mapsto tM'(t) - M(t)$, but only the equivalence of $\mathbb E \| (x_iX_i)_{i=1}^n \|_2$ and $\|x\|_M$, then it is not hard to see that $t\mapsto tM'(t) - M(t)$ already had to be $2$-concave (see Proposition \ref{thm:general}). Therefore, it seems that the condition is natural and not ``too far'' from the $2$-concavity of $M$.

Our main result is the following:

\begin{thm} \label{main}
Let $1<p<\infty$ and $M\in\mathcal{C}^3$ be an Orlicz function with $M'(0)=0$ and $M''(T)=0$ for $T=M^{-1}(1)$.
Assume the normalization $\int_0^\infty x\dif M'(x)=1$ and that $M|_{[T,\infty)}$ is linear.
Moreover, assume that for all $x>0$
\[
f_X(x)=\Big(1-\frac{2}{p}\Big)\frac{1}{x^3}M''\Big(\frac{1}{x}\Big)-\frac{1}{px^4}M'''\Big(\frac{1}{x}\Big)\geq 0.
\]
Then $f_X$ is a probability density and for all $x\in\R^n$,
\[
c_1(p-1)^{1/p}\|x\|_{M} \leq \E \|(x_i X_i)_{i=1}^n\|_p \leq c_2\|x\|_{M},
\]
where $c_1,c_2$ are positive absolute constants and $X_1,\dots,X_n$ are iid with density $f_X$.
\end{thm}

If $M$ is not normalized, we can divide the function $f_X$ by $\int_0^\infty x\dif M'(x)$ to obtain a probability density and  the statement of the theorem is  true with constants depending on $p$ and $M$. Due to the definition of the Orlicz norm, its value is uniquely determined by the values of the function $M$ on the interval $[0,M^{-1}(1)]$. Hence, it is no restriction to extend $M$ linearly.
If $p=2$, this immediately yields the desired embedding of Orlicz spaces into $L_1$ (see Corollary \ref{embedding}). In fact, we will prove the case $p=\infty$ first, which will then imply the result for arbitrary $\ell_p$-norms.

\section{Preliminaries and Notation}\label{Preliminaries}

A convex function $M:[0,\infty)\to[0,\infty)$ where $M(0)=0$ and $M(t)>0$ for $t>0$ is called an \emph{Orlicz function}. 
The $n$-dimensional \emph{Orlicz space} $\ell_M^n$ is $\R^n$ equipped with the norm
  \begin{equation}\label{def:orlicznorm}
    \Vert{x}\Vert_M = \inf \Big\{ \rho>0 \,:\, \sum_{i=1}^n M\left(|x_i|/\rho\right) \leq 1 \Big\}.
  \end{equation}
In case $M(t)=t^p$, $1\leq p<\infty$ we just have $\ell_M^n = \ell_p^n$, i.e., $\Vert\cdot\Vert_M=\Vert\cdot\Vert_p$.
Given Orlicz functions $M_1,\dots,M_n$, we define the corresponding \emph{Musielak-Orlicz function} as $\mathbb M = (M_1,\dots,M_n)$ and the $n$-dimensional \emph{Musielak-Orlicz space} $\ell_{\mathbb M}^n$ is $\R^n$ equipped with the norm
 \[
    \Vert{x}\Vert_{\mathbb M} = \inf \Big\{ \rho>0 \,:\, \sum_{i=1}^n M_i\left(|x_i|/\rho\right) \leq 1 \Big\}.
 \]
If $M_i=M$ for all $i=1,\ldots,n$, then $\ell_{\mathbb M}^n = \ell_M^n$. 
We say that two Orlicz functions $M$ and $N$ are equivalent
if there are positive constants $a$ and $b$ such that for all
$t\geq0$
\[
 a^{-1}M( b^{-1}t) \leq N(t) \leq aM(bt).
\]
If two Orlicz functions are equivalent so are their norms. An Orlicz function is said to be \emph{$p$-concave} for some $1\leq p<\infty$, if $t\mapsto M(t^{1/p})$ is a concave function. 
We say that an Orlicz function $M$ is \emph{normalized} if 
\[
\int_0^\infty x\dif M'(x)=1.
\]
Note also that, if two Orlicz functions are equivalent in a neighborhood of zero, then the corresponding sequence spaces already coincide \cite[Proposition 4.a.5]{LT1}. For a detailed and thorough introduction to the theory of Orlicz spaces we refer the reader to \cite{KR}, \cite{RR} or \cite{LT1,LT2} and to \cite{M} in the case of Musielak-Orlicz spaces.  

Let $X$ and $Y$ be isomorphic Banach spaces. We say that they are
\emph{$C$-isomorphic} if there is an isomorphism $T:X\rightarrow Y$ with
$\|T\|\|T^{-1}\|\leq C$.
We define the Banach-Mazur distance of $X$ and $Y$ by
    \[
      d(X,Y) = \inf\left\{ \|T\|\|T^{-1}\| \,:\, T\in L(X,Y) ~ \hbox{isomorphism} \right\}.
    \]
 Let $(X_n)_n$ be a sequence of $n$-dimensional normed spaces and let $Z$ also be a normed space. If there exists a constant $C>0$, such that for all $n\in\N$ there exists a normed space $Y_n \subseteq Z$ with $\dim(Y_n)=n$ and $d(X_n,Y_n)\leq C$, then we say $(X_n)_n$ \emph{embeds uniformly} into $Z$. The beautiful monograph \cite{TJ} gives a detailed introduction to the concept of Banach-Mazur distances.

We will use the notation $A\sim B$ to indicate the existence of two positive absolute constants $c_1,c_2$ such that $c_1A\leq B\leq c_2A$. Similarly, we define the symbol $\lesssim$. We write $\sim_p$, with some positive constant $p$, to indicate that the constants $c_1$ and $c_2$ depend on $p$.  $c_1, c_2, c, C,\dots$ will always denote positive absolute constants whose value may change from line to line.

By $L_1$ we denote the $L_1$ space on the unit interval $[0,1]$ with Lebesgue measure.

We write $f\in \mathcal C^k$ for some $k\in\N$, whenever the function $f$ is $k$ times continuously differentiable and $\mathcal C^k(a,b)$ for $\mathcal C^k((a,b))$.

The following theorem was obtained in \cite{GLSW2} and provides a formula for the Orlicz function $M$ provided that we know the distribution of $X$:

\begin{thm}{\rm(}\cite[Lemma 5.2]{GLSW2}{\rm)}.\label{thm:orlicz}
Let $X_1,\dots X_n$ be iid integrable random variables. For all $s\geq 0$ define
\[
M(s)=\int_0^s \int_{1/t\leq |X_1|} |X_1|\,\dif \mathbb P \dif t.
\]
Then, for all $x=(x_i)_{i=1}^n\in \mathbb{R}^n$,
\[
c_1\|x\|_M\leq \mathbb{E}\max_{1\leq i\leq n}|x_iX_i|\leq c_2\|x\|_M,
\]
where $c_1,c_2$ are absolute constants independent of the distribution of $X_1$.
\end{thm}

Obviously, the function
\begin{equation}\label{EQU Orlicz function M}
  M(s)=\int_0^s \int_{1/t\leq |X_1|} |X_1| \, \dif\mathbb P \, \dif t
\end{equation}
is non-negative and convex, since $\int_{1/t\leq |X|}|X| \,d\mathbb P$ is increasing in $t$. Furthermore, we have that $M$ is continuous, differentiable and $M(0)=M'(0)=0$.

Note that, in fact, Theorem \ref{thm:orlicz} is true for Musielak-Orlicz spaces when we do not assume the random variables to be identically distributed: 

\begin{thm}\label{thm:orliczgeneral}
Let $X_1,\dots X_n$ be independent integrable random variables. For all $s\geq 0$ and all $j=1,\ldots,n$ define
\[
M_j(s)=\int_0^s \int_{1/t\leq |X_j|} |X_j|\,\dif \mathbb P \dif t.
\]
Then, for all $x=(x_i)_{i=1}^n\in \mathbb{R}^n$,
\[
c_1\|x\|_{\mathbb M}\leq \mathbb{E}\max_{1\leq i\leq n}|x_iX_i|\leq c_2\|x\|_{\mathbb M},
\]
where $c_1,c_2$ are absolute constants and $\mathbb M = (M_1,\dots,M_n)$.
\end{thm}
A proof in the case of averages over permutations can be found in \cite{P} and can be generalized to our setting by a straightforward adaption of the proof of Theorem \ref{thm:orlicz}.

\begin{rmk}
Because of Theorem \ref{thm:orliczgeneral}, all results presented in this paper hold in the more general setting of Musielak-Orlicz spaces, but for notational convenience we state them only for Orlicz spaces.
\end{rmk}

\begin{rmk}
If $M$ is an Orlicz function such that $M\in \mathcal C^3$, then for $t\mapsto tM'(t)-M(t)$ to be $2$-concave is equivalent to $M'''\leq 0$. Therefore, and for the sake of convenience, we will later assume $M'''\leq 0$, but might still talk about the $2$-concavity of $t\mapsto tM'(t)-M(t)$ at the same time.
\end{rmk}

We will also need a result from \cite{PR} about the generating distribution of $\ell_p$-norms. We recall that the density of a $\log \gamma_{1,p}$ distributed random variable $\xi$ with parameters $p>0$ is given by
\[
f_{\xi}(x) = px^{-p-1}\mathbbm 1_{[1,\infty)}(x).
\]
Note also that for all $x>0$
\[
\mathbb P \left(\xi \geq x \right) = \min(1,x^{-p}).
\]

\begin{thm}{\rm(}\cite[Theorem 3.1]{PR}{\rm)}.\label{THM_lp_normen}
  Let $p>1$ and $\xi_1,...,\xi_n$ be iid copies of a $\log \gamma_{1,p}$ distributed random variable $\xi$. Then, for all $x\in\R^n$,
    \[
      c_1 \|x\|_{p} \leq \mathbb E \max_{1\leq i \leq n} | x_i\xi_i | \leq \frac{c_2}{(p-1)^{1/p}} \|x\|_{p},
    \]
    where $c_1,c_2$ are positive absolute constants.
 \end{thm}

 Recall the following well-known theorem about the existence of independent random variables corresponding to given distributions:

 \begin{thm}{\rm(}\cite[Theorem 20.4]{B}{\rm)}.\label{thm_measure}
 Let $(\mu_j)_j$ be a finite or infinite sequence of probability measure on the real line. Then there exists an independent sequence of random variables $(\xi_j)_j$ defined on the probability space $([0,1],\mathfrak{B}_\R,\lambda)$, with Borel $\sigma$-algebra $\mathfrak{B}_\R$ and Lebesgue measure $\lambda$, so that the distribution of $\xi_j$ is $\mu_j$.
 \end{thm}

\section{A simple Representation Result}\label{sec:simple}

In this section we prove a result of the same spirit as Theorem \ref{thm:orlicz}, where we replace the $\ell_\infty$-norm by some $\ell_p$-norm for $1< p <\infty $. This is a special case of Theorem 1 in \cite{GLSW1} with $N(t)=t^p$. There it seems unclear how to determine the ``precise'' form of the Orlicz function that appears. Of course, this is somehow unsatisfactory and, therefore, we provide a result that produces a ``simple'' representation of this Orlicz function. Observe also that the following result, which is a consequence of Theorems \ref{thm:orlicz} and \ref{THM_lp_normen}, corresponds to the discrete results recently obtained in \cite{PS}.

\begin{thm} \label{thm:orlicz_p_norm}
Let $1<p<\infty$, $X_1,\ldots,X_n$ be iid integrable random variables. For all $s\geq0$ define
\[
M(s) = \frac{p}{p-1}\int_0^s\left( \int_{|X_1| \leq \frac{1}{t}} t^{p-1} \left| X_1 \right|^p \dif \Pro + \int_{ |X_1| > 1/t}|X_1| \dif \mathbb P \right)\dif t .
\]
Then, for all $x\in\R^n$,
\[
 c_1 (p-1)^{1/p} \| x \|_M \leq \mathbb E \| (x_iX_i)_{i=1}^n \|_p \leq c_2 \| x \|_M,
\]
where $c_1,c_2,$ are positive absolute constants.
\end{thm}
\begin{proof}
Let $X_1,\dots,X_n$ be defined on $(\Omega_1,\Pro_1)$ and let $\xi_1,\dots,\xi_n$ be independent copies of a $\log\gamma_{1,p}$ distributed random variable $\xi$, say on $(\Omega_2,\mathbb P_2)$. Then, by Theorem \ref{THM_lp_normen},
\[
\mathbb E_{\Omega_1} \|(x_iX_i)_{i=1}^n\|_p\lesssim \mathbb E_{\Omega_1} \mathbb E_{\Omega_2} \max_{1\leq i \leq n}|x_i X_i\xi_i| \lesssim (p-1)^{-1/p} \mathbb E_{\Omega_1} \|(x_iX_i)_{i=1}^n\|_p,
\]
holds for all $x\in\R^n$.
On the other hand, by Theorem \ref{thm:orlicz},
\[
\mathbb E_{\Omega_1} \mathbb E_{\Omega_2} \max_{1\leq i \leq n}|x_i X_i\xi_i| \sim \|x\|_{M}
\]
for all $x\in\R^n$, where
\[
M(s)=\int_0^s \int_{1/t \leq |X_1\xi|} |X_1\xi|\,\dif \mathbb P \dif t.
\]
For $t>0$ and $\omega_1\in\Omega_1$ define
\[
I_{\omega_1} := \left\{\omega_2\in\Omega_2 \,:\, t |\xi(\omega_2) X_1(\omega_1)| \geq 1\right\}.
\]
Now, we observe that
\begin{align*}
M(s) & = \int_0^s \int_{\Omega_1} \int_{I_{\omega_1}} |X_1(\omega_1)\xi(\omega_2)|\,\dif \mathbb P_2(\omega_2) \dif \mathbb P_1(\omega_1) \dif t \\
& = \int_0^s \int_{\Omega_1} |X_1(\omega_1)| \int_{I_{\omega_1}}  |\xi(\omega_2)|\,\dif \mathbb P_2(\omega_2) \dif \mathbb P_1(\omega_1) \dif t .\\
\end{align*}
Let us take a closer look at the inner integral. Fix $t>0$ and $\omega_1\in\Omega_1$ and recall that the density of $\xi$ is
\[
f_{\xi}(x) = px^{-p-1}\mathbbm 1_{[1,\infty)}(x).
\]
Therefore, if $t|X_1(\omega_1)| \leq 1$,
\[
 \int_{I_{\omega_1}} |\xi(\omega_2)|\,\dif \mathbb P_2(\omega_2) = p \int_{\{z\,:\, zt|X_1(\omega_1)| \geq 1\}} z^{-p} \dif z = \frac{p}{p-1} (t |X_1|)^{p-1}.
\]
Now assume that $t|X_1(\omega_1)| \geq 1$. Then we get
\[
 \int_{I_{\omega_1}} |\xi(\omega_2)|\,\dif \mathbb P_2(\omega_2) = \mathbb E |\xi| = \frac{p}{p-1}.
\]
Hence, by splitting the integral over $\Omega_1$, for fixed $t$ we have
\begin{eqnarray*}
&&\int_{\Omega_1}  \int_{I_{\omega_1}} |X_1(\omega_1)\xi(\omega_2)|\,\dif \mathbb P_2(\omega_2) \dif \mathbb P_1(\omega_1)\\
& = &\frac{p}{p-1} \int_{ |X_1| \leq 1/t} t^{p-1} \left| X_1 \right|^p\dif \mathbb P_1(\omega_1) + \frac{p}{p-1} \int_{ |X_1| > 1/t}|X_1|\dif \mathbb P_1(\omega_1).
\end{eqnarray*}
This implies the result.
\end{proof}

Note that by Fubini's theorem,
\begin{align*}
\int_0^s \int_{0}^{\frac{1}{t}} t^{p-1}|x|^p \dif \Pro_{X_1}(x) \dif t & = \int_{\frac{1}{s}}^\infty |x|^p\int_{0}^{|x|^{-1}} t^{p-1} \dif t \dif \Pro_{X_1}(x) \\
& = \frac{1}{p} \Pro\left( |X_1| \geq s^{-1} \right)   \leq \frac{1}{p},
\end{align*}
and, hence, the limit case in Theorem \ref{thm:orlicz_p_norm} for $p\to \infty$ coincides with Theorem \ref{thm:orlicz}.

Observe also that Theorem \ref{thm:orlicz_p_norm} provides a natural candidate for the probability density that appears in Theorem \ref{main}:

If the random variables $|X_1|,\dots,|X_n|$ have a density $f_X$, then
\[
M''(s) = ps^{p-2}\int_{0}^{s^{-1}} x^pf_X(x) \dif x,
\]
that is,
\[
\int_{0}^{s^{-1}} x^pf_X(x) \dif x = \frac{1}{p}s^{2-p}M''(s).
\]
Therefore, differentiating once again,
\[
f_X(s^{-1}) = \left( 1-\frac{2}{p}\right) s^3M''(s) - \frac{1}{p}s^4M'''(s).
\]

In the following section we will prove Theorem \ref{main} in the case $p=\infty$. We then reduce the case of general $p$ to the case $p=\infty$ in Section \ref{general_p}.

\section{The case of the $\ell_\infty$-norm}\label{sec:infinity}

To obtain the case of $\ell_p$-norms it is enough to settle the question for the $\ell_\infty$-norm. We will give a short explanation of that fact:

Assume that $N$ is an arbitrary Orlicz function and we know how to choose a distribution (depending on $N$) so that, if $\xi_1,\ldots,\xi_n$ are independent random variables distributed according to that law, then, for all $x=(x_i)_{i=1}^n\in\R^n$,
\[
\mathbb E \max_{1\leq i \leq n} \left| x_i \xi_i\right| \sim \| x\|_N.
\]
Now, let $M$ be the normalized Orlicz function given in Theorem \ref{main}. We want to find a distribution and independent random variables $X_1,\ldots,X_n$ defined on a measure spaces $(\Omega_1,\Pro_1)$ distributed according to this such that
\begin{equation}\label{equ_ell_p_norm}
\mathbb E_{\Omega_1} \| (x_i X_i)_{i=1}^n\|_p \sim_p \|x\|_M.
\end{equation}
Of course, we can find a distribution and accordingly distributed independent random variables $Z_1,\ldots,Z_n$ so that
\[
\mathbb E  \max_{1\leq i \leq n} \left| x_iZ_i \right| \sim \| x \|_M,
\]
since we can just take $N=M$.
On the other hand, observe that
\[
\mathbb E_{\Omega_1} \| (x_i X_i)_{i=1}^n\|_p \sim_p \mathbb E_{\Omega_1} \mathbb E_{\Omega_2} \max_{1\leq i \leq n} \left| x_iX_iY_i\right|,
\]
where we get the distribution of the independent random variables $Y_1,\ldots,Y_n$, say on $(\Omega_2,\Pro_2)$, by choosing $N(t)=t^p$. So, for all $x=(x_i)_{i=1}^n \in\R^n$,
\[
\mathbb E \max_{1\leq i \leq n} \left| x_iZ_i \right| \sim_p \|x\|_M \sim \mathbb E_{\Omega_1} \mathbb E_{\Omega_2} \max_{1\leq i \leq n} \left| x_iX_iY_i\right|.
\]
Therefore, to obtain (\ref{equ_ell_p_norm}), we just have to choose the distribution of $X_1,\ldots,X_n$ so that $X_1Y_1\stackrel{\mathcal D}{=} Z_1$.
Of course, here the distribution of $Z$ and $Y$ is known.

Before we continue, we observe that the transformation formula for integrals yields the following substitution rule for Stieltjes integrals:
\begin{equation}\label{eq:subst}
\int_a^b f\circ u  \dif(F\circ u) = \int_{u(a)}^{u(b)} f  \dif F,
\end{equation}
where $f$ is an arbitrary measurable function, $F$ is a non-decreasing function and $u$ is monotone on the interval $[a,b]$.

The following result is the converse to Theorem \ref{thm:orlicz}:

\begin{proposition}\label{PRO_inverse_maximum}
Let $M$ be a normalized Orlicz function with $M'(0)=0$. Let $X_1,\ldots,X_n$ are independent copies of a random variable $X$ with distribution
\begin{equation}\label{eq:distrX}
\mathbb P(X\leq t)=\int_{[1/t,\infty)} s\dif M'(s),\quad t> 0.
\end{equation}
Then, for all $x=(x_i)_{i=1}^n\in \mathbb{R}^n$,
\[
c_1\|x\|_M\leq \mathbb{E}\max_{1\leq i\leq n}|x_iX_i|\leq c_2\|x\|_M,
\]
where $c_1,c_2$ are constants independent of the Orlicz function $M$.
\end{proposition}
\begin{proof}
We first observe that for an arbitrary random variable $X$ which is $\geq 0$ a.s., we have by \eqref{eq:subst}
\begin{equation*}
F_X(t):=\mathbb P(X\leq t)=\int_{(0,t]} \dif F_X(s)=-\int_{[1/t,\infty)}  \dif (F_X\circ u)(s),
\end{equation*}
where $u(s)=1/s$. If the distribution of $X$ is given by \eqref{eq:distrX}, we obtain
\[
\dif(F_X\circ u)(s)=-s\dif M'(s).
\]
Now we obtain, again by \eqref{eq:subst} and this identity
\begin{align*}
\int_0^s \int_{[1/t,\infty)} x\dif F_X(x)\dif t &= -\int_0^s \int_{(0,t]} \frac{1}{x} \dif(F_X\circ u)(x)\dif t \\
&= \int_0^s \int_{(0,t]} \dif M'(x)\dif t \\
&= M(s).
\end{align*}
The assertion of the theorem is now a consequence of Theorem \ref{thm:orlicz}.
\end{proof}

\begin{rmk}
The assumption that $M$ is normalized, i.e., $\int_0^\infty x\,\dif M'(x)=1$, assures us that the constants do not depend on $M$.
Note also that, as an immediate consequence of Proposition \ref{PRO_inverse_maximum}, by the integration by parts rule for Stieltjes integrals we obtain
\begin{equation}\label{equ_tail_distribution_function_of_X}
\mathbb P\left( X > t\right) = \int_{0}^{\frac{1}{t}} s \, \dif M'(s) = \frac{1}{t}M'\left(\frac{1}{t}\right) - M\left(\frac{1}{t}\right)
\end{equation}
for any $t>0$. If $M$ is ``sufficiently smooth'', we get that the density $f_X$ of $X$ is given by
\[
f_X(t)={t^{-3}}M''(t^{-1}).
\]
\end{rmk}

\begin{rmk}
To generate an $\ell_p$-norm in Proposition \ref{PRO_inverse_maximum}, i.e., to consider the case $M(t)=t^p$, one needs to pass to an equivalent Orlicz function so that the normalization condition is satisfied. The function $\widetilde M$ with $\widetilde M(t) = t^p$ on $[0, (p-1)^{-1/p}]$ which is then extended linearly does the trick.
\end{rmk}

\section{The case of $\ell_p$-norms}\label{general_p}

We will now prove the result which will then imply the main result, Theorem \ref{main}. Of course, in the proposition we could also assume $M\in\mathcal C^3$, but $M\in\mathcal C^2$ so that $M''$ is absolutely continuous on each compact subinterval of $(0,\infty)$ is sufficient.

\begin{proposition}\label{thm:p}
Let $M\in\mathcal C^2(0,\infty)$ be a normalized Orlicz function and $M''$ be absolutely continuous on each compact subinterval of $(0,\infty)$. Assume that $M'(0)=0=M''(T)$ for $T=M^{-1}(1)$ and that $M|_{[T,\infty)}$ is linear.
Let $1< p < \infty$ and $X, Y$ be two independent random variables distributed according the laws
\begin{align*}
\Pro(Y\geq y)&=\min (1,y^{-p})\quad\text{and } \\
\Pro(X\geq x)&=-M\Big(\frac{1}{x}\Big)+\frac{1}{x} M'\Big(\frac{1}{x}\Big)-\frac{1}{px^2}M''\Big(\frac{1}{x}\Big).
\end{align*}
Then the tail distribution function of $XY$ is
\begin{equation}\label{eq:probXY}
\Pro(XY\geq z)=\frac{1}{z}M'\Big(\frac{1}{z}\Big)-M\Big(\frac{1}{z}\Big), \quad z>0.
\end{equation}
\end{proposition}

\begin{proof}
First note that the density function of $X$ is given by
\begin{equation}\label{eq:densityXY}
\begin{aligned}
f_X(x)&=\Big(1-\frac{2}{p}\Big)\frac{1}{x^3}M''\Big(\frac{1}{x}\Big)-\frac{1}{px^4}M'''\Big(\frac{1}{x}\Big), \\
\end{aligned}
\end{equation}
Inserting the expression for $\Pro(Y\geq y)$, we obtain
\begin{equation}
\label{eq:prodxyintermed}
\begin{aligned}
\Pro(XY\geq z)&=\int \mathbbm{1}_{\{XY\geq z\}} \dif \Pro=\int_0^\infty \Pro(Y\geq z/x) f_X(x)\dif x \\
&= \int_0^\infty \min(1,x^p/z^p)f_X(x)\dif x\\
&= \Pro(X\geq z)+z^{-p}\int_0^z x^p f_X(x)\dif x.
\end{aligned}
\end{equation}
Observe that, under the above assumptions and for $z\leq T^{-1}$, $\Pro(X\geq z)=1=z^{-1}M'(z^{-1})-M(z^{-1})$ and $f_X(z)=0$, since $\int_0^\infty x\dif M'(x)=TM'(T)-M(T)=1$. This yields \eqref{eq:probXY} for $z\leq 1/T$.
Thus we now assume $z>1/T$ and continue with calculating the integral $\int_0^z x^p f_X(x)\dif x$. We substitute $u=1/x$ and obtain
\begin{align*}
\int_0^z x^p f_X(x)\dif x &= \int_{z^{-1}}^\infty u^{-p-2} f_X(u^{-1})\dif u \\
&= \int_{z^{-1}}^T \Big(1-\frac{2}{p}\Big)u^{1-p}M''(u)-\frac{u^{2-p}}{p}M'''(u)\dif u.
\end{align*}
Partial integration further yields
\[
\int_0^z x^p f_X(x)\dif x = -\frac{u^{2-p}}{p}M''(u)\Big|_{z^{-1}}^T=\frac{1}{p}z^{p-2} M''(z^{-1}),
\]
since $M''(T)=0$. Combining equation \eqref{eq:prodxyintermed} with this result and the expression for the distribution of $X$, we obtain \eqref{eq:probXY} for $z>1/T$.
\end{proof}

Now we can finally prove our main theorem:
\begin{proof}[Proof of Theorem \ref{main}]
Let $M$ be the given Orlicz function and $(X_i)_{i=1}^n$ the given random variables on a measure space $(\Omega_1,\Pro_1)$. First note that by Proposition \ref{PRO_inverse_maximum} and the remark after it we get
\begin{equation}\label{eq:main1}
\|x\|_M\sim \E \max_{1\leq i\leq n} |x_i Z_i|,
\end{equation}
where $\Pro(Z\geq z)=z^{-1}M'(z^{-1})-M(z^{-1})$.
Secondly, by Theorem \ref{THM_lp_normen},
\begin{equation}\label{eq:main2}
\E_{\Omega_1} \| (x_iX_i)_{i=1}^n \|_p \lesssim \E_{\Omega_1}\E_{\Omega_2} \max_{1\leq i\leq n} |x_iX_iY_i| \lesssim (p-1)^{-1/p} \E_{\Omega_1} \| (x_iX_i)_{i=1}^n \|_p
\end{equation}
where the random variables $(Y_i)_{i=1}^n$, defined on $(\Omega_2,\Pro_2)$, are independent and $\log\gamma_{1,p}$-distributed. Since, by Proposition \ref{thm:p}, $X_1Y_1\stackrel{\mathcal D}{=} Z_1$, we combine \eqref{eq:main1} and \eqref{eq:main2} to obtain the assertion of the theorem.
\end{proof}

In case $p=2$, we obtain the following corollary:

\begin{cor}\label{cor_p=2}
Let $M\in\mathcal{C}^3(0,\infty)$ be a normalized Orlicz function with $M'(0)=0$ and $M'''(x)\leq 0$ for all $x\geq 0$ and assume that $M''(M^{-1}(1))=0$. Then
\begin{equation}\label{equ_density_for_p=2}
f_X(x)=-\frac{1}{2x^4} {M}'''\Big(\frac{1}{x}\Big)
\end{equation}
is a probability density and for all $x\in\R^n$,
\[
c_1\|x\|_{M}\leq\E\|(x_i X_i)_{i=1}^n\|_2\leq c_2\|x\|_{M},
\]
where $c_1,c_2$ are positive absolute constants and $X_1,\dots,X_n$ are iid with density $f_X$.
\end{cor}

Again, the normalization condition $\int_0^\infty y \dif  M'(y)=1$ assures that constants do not depend on $M$ and, in fact, is of the same form as the normalization condition in Theorem 2 from \cite{S}.
Note also that in the proof of Theorem \ref{thm:p} and its corollaries we need that $M''(T)=0$ for $T=M^{-1}(1)$. This, indeed, is no restriction, since Lemma \ref{lem:approx} in Section \ref{SEC_Appendix} shows that for any $2$-concave Orlicz function we can assume that $M''(T)=0$, otherwise we pass to an equivalent Orlicz function which has this property. Recall also that every Orlicz function which satisfies $M'''\leq 0$ is already $2$-concave. The authors do not know whether for an Orlicz function $M$ to be $2$-concave is equivalent (up to equivalent Orlicz functions) to have non-positive third derivative.

\begin{rmk}
Note that another proof of Corollary \ref{cor_p=2} via a Choquet-type representation theorem in the spirit of Lemma 7 in \cite{S} also yields the condition that the function $z\mapsto zM'\left(z\right) - M\left(z\right)$ has to be $2$-concave (or equivalently $M'''\leq 0$).
\end{rmk}

\section{Orlicz spaces that are isomorphic to subspaces of $L_1$}

As we will see, it is an easy consequence of Corollary \ref{cor_p=2} that the sequence of Orlicz spaces $\ell_M^n$, $n\in\N$, where $t\mapsto tM'(t)-M(t)$ is $2$-concave, embeds uniformly into $L_1$. Although we need $t\mapsto tM'(t)-M(t)$ to be a $2$-concave function, which seems a bit stronger than to assume that $M$ is $2$-concave, the simplicity of the representation (\ref{equ_density_for_p=2}) of the density that we need in our embedding has a strong advantage over the representation in Theorem 2  in \cite{S}, since it is much easier to handle.

We obtain the following result:

\begin{cor}\label{embedding}
Let $M$ be a normalized Orlicz function so that $M'(0)=0$ and $M'''\leq 0$. Then there exists a positive absolute constant $C$ (independent of $M$) such that for all $n\in\N$ there is a subspace $Y_n$ of $L_1$ with $\dim(Y_n)=n$ and
\[
d(\ell_M^n,Y_n) \leq C,
\]
i.e., $(\ell_M^n)_n$ embeds uniformly into $L_1$.
\end{cor}
\begin{proof}
The proof is a simple consequence of Corollary \ref{cor_p=2}, Khintchine's inequality and Theorem \ref{thm_measure}. Given $n\in\N$, we let $\mu_1=\dots=\mu_n$ be the distribution of Rademacher functions, that is,
\[
\mu_i(\{1\})=\mu_i(\{-1\})=1/2,\quad 1\leq i\leq n.
\]
Additionally, we let $\mu_{n+1}=\dots = \mu_{2n}$ be the distribution of $X_i$ given in Corollary  \ref{cor_p=2}. Then we apply Theorem \ref{thm_measure} to the finite sequence $(\mu_i)_{i=1}^{2n}$ of probability measures to get independent random variables $r_1,\dots,r_n,X_1,\dots X_n$ defined on the unit interval $[0,1]$ such that the distribution of $r_i$ is $\mu_i$ and the distribution of $X_i$ is $\mu_{n+i}$ for all $1\leq i\leq n$. Then the asserted isomorphism is given by
\[
\Psi_n:\ell_M^n \to L_1[0,1], \quad a\mapsto \sum_{i=1}^n a_i r_i(\cdot) X_i(\cdot).
\]
Thus, applying Khintchine's inequality, for any $a=(a_i)_{i=1}^n\in\R^n$,
\begin{align*}
\| \Psi_n(a) \|_{L_1} & = \int_0^1 \Big| \sum_{i=1}^n a_i r_i(t) X_i(t)\Big| \dif t  \\
& = \int_{\R^n} \int_{\{-1,1\}^n} \Big| \sum_{i=1}^n a_i \varepsilon_i x_i  \Big|  \dif(\mu_1\otimes\dots\otimes\mu_n)(\varepsilon)\dif (\mu_{n+1}\otimes\cdots\otimes\mu_{2n})(x)\\
& \sim \int_{\R^n} \Big( \sum_{i=1}^n |a_i x_i|^2 \Big)^{1/2} \dif (\mu_{n+1}\otimes\dots\otimes\mu_{2n})(x) \\
& = \int_{[0,1]} \Big( \sum_{i=1}^n |a_i X_i(t)|^2 \Big)^{1/2} \dif t \\
& \sim \|a\|_M,
\end{align*}
where we used Corollary \ref{cor_p=2} in the last step.
\end{proof}

\section{The general result}

Following the ideas described in Section \ref{sec:infinity}, we now generalize our results to find an inequality of the form
\[
\frac{1}{C} \|x\|_M \leq \mathbb E \|(x_iX_i)_{i=1}^n\|_N \leq C \|x\|_M
\]
for a general Orlicz function\ $N$. For each normalized Orlicz function $L$, we write
\[\overline{F}_L(t)=\int_{0}^{1/t}s \dif L'(s) = \frac{1}{t}L'\left(\frac{1}{t}\right) - L\left(\frac{1}{t}\right)\]
and call this function the tail distribution function associated to $L$, motivated by Proposition \ref{PRO_inverse_maximum} and equation \eqref{equ_tail_distribution_function_of_X}.

\begin{proposition}\label{thm:general}
Let $M,N$ be normalized Orlicz functions with $M'(0)=N'(0)=0$.
\begin{enumerate}[(i)]
\item If there exists a probability measure $\mu$ on $(0,\infty)$ such that
\begin{equation}\label{eq:mult_convolution}
\overline F_M(t)=\int_{(0,\infty)} \overline F_N(t/x) \dif \mu(x),
\end{equation}
then, for all $x=(x_i)_{i=1}^n\in \mathbb{R}^n$,
\[
c_1\|x\|_M\leq \mathbb{E}\|(x_iX_i)_{i=1}^n\|_N\leq c_2\|x\|_M,
\]
where $c_1,c_2$ are positive absolute constants and $X_1,\dots,X_n$ are iid random variables with distribution $\mu$.
\item If there exist iid random variables $X_1,\dots,X_n$ with distribution $\mu$ on $(0,\infty)$ such that 
\[
c_1\|x\|_M\leq \mathbb{E}\|(x_iX_i)_{i=1}^n\|_N\leq c_2\|x\|_M,
\]
where $c_1,c_2$ are positive absolute constants, then there exists an Orlicz function $\widetilde{M}$  equivalent to $M$ such that
\begin{equation*}
\overline F_{\widetilde M}(t)=\int_{(0,\infty)} \overline F_N(t/x) \dif \mu(x).
\end{equation*}
\end{enumerate}
\end{proposition}

\begin{proof} 
(i): Note that condition \eqref{eq:mult_convolution} guarantees that we can follow the line of argument in the proof of Theorem \ref{main}. Indeed, we
choose independent sequences of iid random variables $(Z_1,\dots,Z_n)$ defined on $(\Omega_1,\Pro_1)$ and $(Y_1,\dots,Y_n)$ defined on $(\Omega_2,\Pro_2)$ with tail distribution functions $\overline F_M$ and $\overline F_N$, respectively. By Proposition \ref{PRO_inverse_maximum} we have
\[\|x\|_M\sim \E_{\Omega_1} \max_{1\leq i\leq n} |x_i Z_i| \quad \textrm{and} \quad\|x\|_N\sim \E_{\Omega_2} \max_{1\leq i\leq n} |x_i Y_i|\]
for all $(x_i)_{i=1}^n\in\R^n$.
By \eqref{eq:mult_convolution}, $X_1Y_1\stackrel{\mathcal D}{=} Z_1$, since for all $t>0$
\begin{equation}\label{eq:productdistribution}
\begin{aligned}
\Pro(Z_1> t)&=\overline F_M(t)=\int_{(0,\infty)} \overline F_N(t/x)\dif \mu(x) \\
&=\int_{(0,\infty)}\Pro(xY_1> t)\dif \mu(x)=\Pro(X_1Y_1> t).
\end{aligned}
\end{equation}
Therefore,
\begin{align*}
\|x\|_M &\sim \E_{\Omega_1} \max_{1\leq i\leq n} |x_i Z_i|=\E_{\Omega}\E_{\Omega_2} \max_{1\leq i\leq n} |x_i X_iY_i|\\
&=\int_{\Omega}\E_{\Omega_2} \max_{1\leq i\leq n} |x_i X_i(\omega)Y_i|\dif \Pro(\omega)\\
&\sim \int_{\Omega}\|(x_i X_i(\omega))_{i=1}^n\|_N\dif \Pro(\omega) \\
&=\E_{\Omega}\|(x_i X_i)_{i=1}^n\|_N.
\end{align*}
(ii): Assume that
\begin{align*}
\E\|(x_i X_i)_{i=1}^n\|_N \sim \|x\|_M
\end{align*}
for iid random variables $X_1,\dots,X_n$ with distribution $\mu$. Define the tail distribution function $\overline{F}$ by
\begin{equation*}
\overline F(t)=\int_{(0,\infty)} \overline F_N(t/x) \dif \mu(x)
\end{equation*}
and choose a sequence of iid random variables $(Z_1,\dots,Z_n)$ defined on $(\Omega_1,\Pro_1)$ with tail distribution function $\overline F$ and sequence $(Y_1,\dots,Y_n)$ independent of $(X_1,\dots,X_n)$ defined on $(\Omega_2,\Pro_2)$ with tail distribution function $\overline F_N$. By construction, $Z_i$ has the same distribution as $X_iY_i,i=1,\dots,n$. Now define the Orlicz function $\widetilde{M}$ by
\[
\widetilde M(s)=\int_0^s \int_{1/t\leq |Z_1|} |Z_1|\,\dif \mathbb P_1 \dif t.
\]
By Theorem \ref{thm:orlicz}, $\|x\|_{\widetilde M}\sim \E_{\Omega_1} \max_{1\leq i\leq n} |x_i Z_i|$ and, therefore, we obtain
\begin{align*}
\|x\|_M&\sim \E_{\Omega}\|(x_i X_i)_{i=1}^n\|_N= \int_{\Omega}\|(x_i X_i(\omega))_{i=1}^n\|_N\dif \Pro(\omega)\\
&\sim\int_{\Omega}\E_{\Omega_2} \max_{1\leq i\leq n} |x_i X_i(\omega)Y_i|\dif \Pro(\omega)=\E_{\Omega}\E_{\Omega_2} \max_{1\leq i\leq n} |x_i X_iY_i|\\
&=\E_{\Omega_1} \max_{1\leq i\leq n} |x_i Z_i|\sim \|x\|_{\widetilde M}.
\end{align*}
Thus, $M$ and $\widetilde M$ are equivalent  \cite[Proposition 4.a.5]{LT1}.
\end{proof}

Condition \eqref{eq:mult_convolution} seems hard to check for general Orlicz functions $M$ and $N$. However, in the special case that we have $N(t)=t^2$ on $[0,1]$ which is then extended linearly, condition \eqref{eq:mult_convolution} is equivalent to the positivity of the function $f_X$ in \eqref{equ_density_for_p=2}. Indeed,
\[\overline F_M(t)=\int_{(0,\infty)} \overline F_N(t/x)\dif \mu(x)=\int_{(0,\infty)} \min(1,x^2/t^2)\dif \mu(x).\]
Note that
\[\int_{(0,\infty)} \min(1,x^2z^2)\dif \mu(x)=\overline F_M\left(1/z\right)=zM'\left(z\right) - M\left(z\right)\]
is obviously a $2$-concave function in $z$ as an average over such functions, in correspondence with the discussion before. On the other hand, Corollary \ref{cor_p=2} can be restated in the following form that shows that the converse is also true: if $z\mapsto zM'\left(z\right) - M\left(z\right)$ is $2$-concave under the conditions stated in Corollary \ref{cor_p=2}, the tail distribution function $\overline F_M$ has a representation of the form \eqref{eq:mult_convolution} and the distribution $\mu$ is explicitly given by the density
\begin{equation*}
f(x)=-\frac{1}{2x^4} {M}'''\Big(\frac{1}{x}\Big).
\end{equation*}

\section{Appendix}\label{SEC_Appendix}

We provide some approximation results for Orlicz functions that we need in this paper and which might be interesting in further applications.

\begin{lemma}
Let $M\in \mathcal{C}^2(0,\infty)$ be an Orlicz function with $M'(0)=0$ and such that $M''$ is decreasing. Then $M$ is $2$-concave.
\end{lemma}
\begin{proof}
Recall that $M$ is $2$-concave if and only if $xM''(x) \leq M'(x)$. For all $\varepsilon\in(0,x)$, there exists $\xi_\varepsilon\in (\varepsilon,x)$ such that
\[
M'(x)=M'(\varepsilon)+(x-\varepsilon)M''(\xi_\varepsilon).
\]
Since $M''$ is decreasing, we get
\[
M'(x)\geq M'(\varepsilon)+(x-\varepsilon)M''(x),
\]
and so, for $\varepsilon\rightarrow 0$, $M'(x)\geq xM''(x)$, which means that $M$ is $2$-concave.
\end{proof}

\begin{lemma}\label{lem:approx}
Let $M\in \mathcal{C}^2(0,M^{-1}(1))$ be an Orlicz function that is linear to the right of $T:=M^{-1}(1)$. Then, for all constants $c>1$, there exists an Orlicz function $N$ such that
\begin{enumerate}
  \item $N''(T) = 0$
  \item \label{it:two} $N(t) \leq M(t) \leq c N(t)$ for all $t\in[0,\infty)$.
\end{enumerate}
Additionally, if $M''$ is decreasing, we can choose $N$ such that $N''$ is decreasing.
\end{lemma}
\begin{proof}
We let $\delta\in(0,1)$ and define $N$ as follows: We set $N(t)=M(t)$ for all $t\leq T(1-\delta)$ and we extend $M$ to $[0,T]$ such that $N''$ is smooth, decreasing, $N''(t)\leq M''(t)$ for $t\in [0,T)$ and $N''(T)=0$. For $t>T$, we define $N$ linearly with the same slope as $M$.

We have to show property \eqref{it:two}. The inequality $N(t)\leq M(t)$ follows from the construction for all $t\in[0,\infty)$. The second inequality is trivial for $t\leq T(1-\delta)$ since for such $t$, $M(t)=N(t)$.
Next, we explore the case $t\in [T(1-\delta),T]$. If we choose $t$ in this interval, by the above definition of $N$,
\begin{align*}
0 &\leq 	M(t)-N(t) \\
  &=		\int_{T(1-\delta)}^t \int_{T(1-\delta)}^s M''(x)-N''(x)\dif x\dif s \\
  &\leq 	T\delta^2 \max_{x\in[T(1-\delta),T]} \big(M''(x)-N''(x)\big) \\
  &\leq		T\delta^2 \max_{x\in[T(1-\delta),T]} M''(x).
\end{align*}
Now we choose $\delta$ such that $T\delta^2 \max_{x\in[T(1-\delta),T]} M''(x)\leq (c-1)M(T(1-\delta))$. This is possible, since $\max_{x\in[T(1-\delta),T]} M''(x)$ is an increasing function of $\delta$ and $M(T(1-\delta))$ is  a decreasing function of $\delta$. Then we obtain for $t\in[T(1-\delta),T]$
\begin{align*}
M(t)	&= 		N(t)+M(t)-N(t) \\
		&\leq 	N(t)+(c-1)M(T(1-\delta)) \\
		&=		N(t)+(c-1)N(T(1-\delta)) \\
		&\leq	cN(t).
\end{align*}
This is property \eqref{it:two} for $t\in [T(1-\delta),T]$. Since for $t\geq T$, the difference $M(t)-N(t)$ is constant by definition of $N$, and the two Orlicz functions $M$ and $N$ are both increasing, the inequality $M(t)\leq cN(t)$ also holds for $t\geq T$ by the following simple calculation:
\begin{align*}
M(t)	&= 		N(t)+M(t)-N(t) \\
		&=		N(t)+M(T)-N(T) \\
		&\leq	N(t)+(c-1)N(T) \\
		&\leq	cN(t).
\end{align*}
This completes the proof.
\end{proof}

Figure \ref{figure1} illustrates the choice of the equivalent Orlicz function in the proof of Lemma \ref{lem:approx} which has the desired properties.

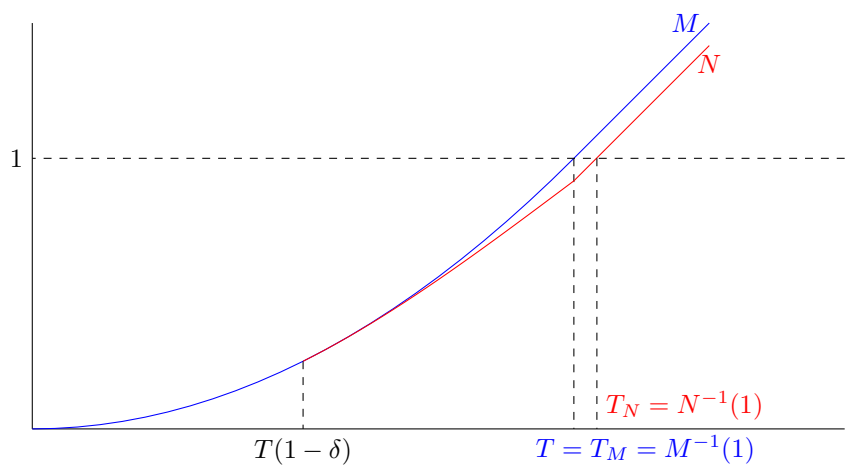
\begin{figure}
\begin{center}
\begin{tikzpicture}[scale=1.8]
\draw (0,0)--(6,0);
\draw (0,0)--(0,3);
\draw[domain=0:4,blue] plot(\x,\x^2/8);
\draw[blue](4,2)--(5,3) node[left,blue]{$M$};
\draw[dashed] (6,2)--(0,2) node[left]{$1$};
\draw[dashed] (4,2)--(4,0);
\draw(3.65,-0.15) node[anchor=west,blue]{$T=T_M=M^{-1}(1)$};
\draw[dashed] (2,0.5)--(2,0) node[below]{$T(1-\delta)$};
\draw[domain=2:4,red] plot(\x,\x^2/4-\x^3/48-\x/4+1/6);
\draw[red](4,1.8333)--(5,2.8333) node[below,red]{$N$};
\draw[dashed] (4.17,2)--(4.17,0) node[anchor=south west,red]{$T_N=N^{-1}(1)$};
\end{tikzpicture}
\end{center}
\caption{Approximation of the Orlicz function $M$}
\label{figure1}
\end{figure}

\begin{rmk}
Let $M$ and $N$ be as in Lemma \ref{lem:approx}.
In order to apply this lemma to Proposition \ref{thm:p}, we have to pass once again to an equivalent Orlicz function $\widetilde{N}$, a multiple of the function $N$ constructed in Lemma \ref{lem:approx} (see Figure \ref{figure1}), to assure $M^{-1}(1)=\widetilde{N}^{-1}(1)$ and, hence, that the function $\widetilde{N}$ is ``smooth'' up to the point $\widetilde{N}^{-1}(1)$.
\end{rmk}

\proof[Acknowledgements]
The last named author would like to thank Gideon Schechtman and Carsten Sch\"utt for helpful discussions.

\bibliographystyle{plain}
\bibliography{distribution_of_random_variables_and_orlicz_function}

\end{document}